\documentclass[12pt]{amsart}

\usepackage{xspace}
\usepackage{graphicx}
\usepackage{longtable}
\usepackage{mathtools}
\usepackage{latexsym}
\usepackage{stmaryrd}
\usepackage{fullpage}
\usepackage{amsfonts,amsmath,amscd,amssymb}
\usepackage{enumitem}
\input{xy}
\xyoption{all}
\xyoption{curve}

\newcommand{\bA}{{\mathbb A}}
\newcommand{\bB}{{\mathbb B}}

\newcommand{\bH}{{\mathbb H}}
\newcommand{\bI}{{\mathbb I}}
\newcommand{\bK}{{\mathbb K}}

\newcommand{\bS}{{\mathbb S}}
\newcommand{\bT}{{\mathbb T}}

\newcommand{\bV}{{\mathbb V}}
\newcommand{\bZ}{{\mathbb Z}}

\newcommand{\fa}{{\mathfrak a}}

\newcommand{\fd}{{\mathfrak d}}
\newcommand{\fe}{{\mathfrak e}}
\newcommand{\ff}{{\mathfrak f}}
\newcommand{\fg}{{\mathfrak g}}

\newcommand{\fh}{{\mathfrak h}}

\newcommand{\fr}{{\mathfrak r}}

\newcommand{\cG}{{\mathcal G}}

\newcommand{\cM}{{\mathcal M}}

\newcommand{\cP}{{\mathcal P}}

\newcommand{\Lie}{{{\mbox{\rm Lie}}}}
\newcommand{\Lip}{{{\mbox{\rm Lie}}^{[p]}}}

\newcommand{\Sym}{{{\mbox{\rm Sym}}}}

\newcommand{\Obs}{{{\mbox{\rm Obs}}}}

\newtheorem{theorem}{Theorem}[section]
\newtheorem{prop}[theorem]{Proposition}
\newtheorem{lemma}[theorem]{Lemma}

\newtheorem{cor}[theorem]{Corollary}
\newtheorem{prob}[theorem]{Problem}

\begin{document}
\title[Centrification]{Centrification of Algebras and Hopf Algebras}
\author{Dmitriy Rumynin}
\email{D.Rumynin@warwick.ac.uk}
\address{Department of Mathematics, University of Warwick, Coventry, CV4 7AL, UK
	\newline
	\hspace*{0.31cm}  Associated member of Laboratory of Algebraic Geometry, National
	Research University Higher School of Economics, Russia}
\thanks{The first author was partially supported within the framework of the HSE University Basic Research Program and the Russian Academic Excellence Project `5--100'}
\author{Matthew Westaway}
\email{M.P.Westaway@warwick.ac.uk}
\address{Department of Mathematics, University of Warwick, Coventry, CV4 7AL, UK}
\date{23/01/2020}
\subjclass[2010]{Primary 16S15, Secondary 16T05}
\keywords{Generators and relations for associative algebras, Gr\"{o}bner-Shirshov basis, deformation, Hopf algebra}

\begin{abstract}
  We investigate a method of construction of central
  deformations of associative algebras, which we call centrification.
We prove some general results in the case of Hopf algebras and provide several examples.
\end{abstract}

\maketitle

This paper has as its inception an investigation into the so-called ``universal Bannai-Ito algebra'' (cf. Section~\ref{s5}).
This algebra is derived from the anticommutator spin algebra (or from the Bannai-Ito algebra $BI^\omega$ for any triple $\omega$ of scalars) by taking its defining relations and
declaring them to be central in the new algebra.
This procedure, which we call {\em centrification}, grew into the main topic of this paper.

We immediately saw that the procedure appears in at least two classical contexts.
A restricted enveloping algebra $U_0$ for a field of characteristic $p>0$ has relations $X^p-X^{[p]}$ that use the $p$-structure on the corresponding restricted Lie algebra.
Declaring these elements central is equivalent to dropping these relations, so centrification yields the universal enveloping algebra $U$ (cf. Section~\ref{ug_from_ug}). More fundamentally, however, the central subalgebra generated by these relations is the ``$p$-centre'' of $U$. This is a subalgebra of great importance in the study of modular Lie algebras. For example, in the reductive case, subject to minor restrictions, the $p$-centre of $U$ is the only part of the centre of $U$ which cannot be seen in characteristic 0.

Even more classical objects
are {\em darstellungsgruppen} in Group Theory.
If $G$ is a perfect group, given by generators and relations as $G\cong F/R$,
its {\em darstellungsgruppe} is defined by $G^c \coloneqq [F,F]/[F,R]$.
It is the universal central extension of $G$ and  the kernel of the map $G^c \rightarrow G$ is the Schur multiplier of $G$. This gives the standard presentation of the Schur multiplier: $(R\cap [F,F])/[F,R]$.
The same construction for a perfect Lie algebra $\ff/\fr$ yields its universal central extension $[\ff,\ff]/[\ff,\fr]\rightarrow \ff/\fr$
(see Section~\ref{lie_alg}).


Perusing the literature, we have found several other examples of centrifications of associative algebras (see, for example, \cite{T}) but no general theory.
The goal of this paper is to summarize some examples of this phenomenon and, more substantively, to develop the foundations of the theory.
The most interesting problem concerns finding mechanisms that control the  behaviour of centrifications -- for instance, to ensure that they
are flat deformations. 
We have discovered several such mechanisms.
One is purely combinatorial in terms of Gr\"{o}bner-Shirshov bases (see Proposition~\ref{GSBasis}). 
Others mix combinatorial techniques with Hopf algebra structure (see Theorems~\ref{Th1}, \ref{Th2} and \ref{Th3}).

It is apparent that centrification is an important source of new associative algebras and central deformations.
We hope that after studying this the reader will also be convinced that centrification is an interesting algebraic construction, worth further research. For example, 
centrification could be utilised to investigate a problem of Terwilliger \cite[Problem 12.1]{T}.

Let us describe the content of the paper.
In Section~\ref{s1} we outline the construction.
We study its Hopf theoretic properties in Section~\ref{s2}.
The standard Hopf algebra structure on the free associative algebra is the subject of investigation in Section~\ref{s3}.
We explain two known examples of centrifications in the final two sections~\ref{s4} and \ref{s5}.

\

\section{Centrification}
\label{s1}
Let $\bK$ be a field. We work with an algebra given by generators and relations
\begin{equation} \label{alg_A}
	\bA=\bK\langle X_i , i\in I | R_j=0, j\in J \rangle \, .
\end{equation}
By its {\em partial centrification} we understand the new algebra
\begin{equation} \label{alg_Ac}
	\bA^{c,J_0} \coloneqq
	\bK\langle X_i , i\in I | X_iR_j=R_jX_i, (i,j)\in I\times J_0; R_j, j\in \overline{J_0} \rangle \, ,
\end{equation}
constructed by a subset $J_0\subseteq J$ where 
$\overline{J_0} \coloneqq J\setminus J_0$. 
The {\em full centrification} is 
$\bA^{c} \coloneqq \bA^{c,J}$. 
In other words, every relation $R_j$, $j\in J_0$,  is no longer a relation but a central element.
By $\overline{X}_i$ or $\overline{R}_j$ we denote the image of the corresponding element in $\bA^{c,J_0}$.
Let $\bZ$ be the subalgebra of $\bA^{c,J_0}$ generated by all $\overline{R}_j$, $j\in J_0$.
Clearly, $\bZ$ is a commutative $\bK$-algebra, central in $\bA^{c,J_0}$, admitting a canonical homomorphism
$$
\bZ \rightarrow \bK, \ \overline{R}_j \mapsto 0
$$
such that $\bA^{c,J_0} \otimes_{\bZ} \bK$ is isomorphic to $\bA$. The following problems about $\bZ$ appear interesting.
\begin{prob}
	Find an algorithm producing the relations of $\bZ$ in the generators $\overline{R}_j$, $j\in J_0$.
\end{prob}
\begin{prob}
	Investigate when $\bZ$ satisfies the standard properties of commutative algebras such as integrality, regularity, being Dedekind or a PID or a UFD, etc.
\end{prob}

We may also observe that if $\bA^{c,J_0}$ is an affine $\bK$-algebra and finitely generated as a $\bZ$-module, then we are already in a familiar situation, namely Hypotheses (H) in \cite[III.1]{BG}. In particular, we can import the following results from \cite{BG}:
\begin{itemize}
	\item $\bZ$ is a finitely-generated $\bK$-algebra,
	\item $\bA^{c,J_0}$ and $\bZ$ are Noetherian PI rings,
	\item all irreducible $\bA^{c,J_0}$-modules are finite-dimensional,
	\item $\bA^{c,J_0}$ is a Jacobson ring,
	\item two maximal ideals of $\bA^{c,J_0}$ belong to the same block if and only if they have the same intersection with $Z(\bA^{c,J_0})$, the centre of $\bA^{c,J_0}$.
\end{itemize}

Thus, the centrification is a universal (in a certain sense) central deformation of $\bA$. Let us consider a third algebra given by generators and relations:
\begin{equation} \label{alg_Az}
	\bA^{z,J_0}= \bZ\langle X_i , i\in I | R_j=\overline{R}_j, j\in J \rangle \, .
\end{equation}
Recall here that the $\overline{R}_j$ for $j\in J$ are elements of the commutative $\bK$-algebra $\bZ$. This is clear from the fact that $\bZ$ is, by definition, generated by the $R_j$ for $j\in J_0$, and the fact that for $j\in\overline{J_0}$ we have $\overline{R}_{j}=0$. The latter observation shows that, for each $j\in\overline{J_0}$, the generators $X_i$, $i\in I$, satisfy the relation $R_j=0$ in both $\bA$ and $\bA^{z,J_0}$. When we wish to distinguish the generators of $\bA^{z,J_0}$ from the generators of $\bA$, we write $\breve{X}_i$ for the generator of $\bA^{z,J_0}$ corresponding to the generator $X_i$ of $\bA$. The following fact is immediate. 
\begin{lemma}
	The natural homomorphism
	$$
	\bA^{c,J_0} \rightarrow \bA^{z,J_0}, \ \ \ \overline{X}_i \mapsto \breve{X}_i
	$$
	is an isomorphism of $\bZ$-algebras.
\end{lemma}

Let $\cM\coloneqq M\langle X_i , i\in I\rangle= \{\mbox{monomials in }X_i, i\in I\}$ be the free monoid, equipped with an admissible order $\succeq$, for instance, the deg-lex ordering. For any element of the free $\bZ$-algebra $A\in \bZ\langle X_i \rangle$, write $\fd(A)\in\cM$ for the leading term of $A$.
It is clear that if the presentation~\eqref{alg_Az} is a Gr\"{o}bner-Shirshov basis, then the presentation~\eqref{alg_A} is a Gr\"{o}bner-Shirshov basis
(cf. \cite[1.3]{Kha} for a review of Gr\"{o}bner-Shirshov bases). 

\begin{prob} \label{main_prob}
	Suppose that the presentation~\eqref{alg_A} is a Gr\"{o}bner-Shirshov basis.
	Under which conditions is the presentation~\eqref{alg_Az} a Gr\"{o}bner-Shirshov basis?
	More generally, when is $\bA^{z,J_0}$ flat as a $\bZ$-module?
\end{prob}

We will give several examples when Problem~\ref{main_prob} has an affirmative answer later in the paper. However, we should expect no such answer in general, as the following analysis of the problem shows.

Without loss of generality we may assume that the monomial $\fd (R_j)$ appears in $R_j$ with coefficient $1$ for all $j\in J$.
Suppose the presentation~\eqref{alg_A} is a Gr\"{o}bner-Shirshov basis. Given $\fd (R_j)=ab, \fd (R_k)=bc\in \cM$ with a non-trivial $b$, we can reduce the composition 
\begin{equation} \label{comp_A}
	R_j\circ_{b}R_k\coloneqq R_jc-aR_k = 
	\sum_{i=1}^{n}\beta_i  u_i {R_{l_i}} v_i \in \bK \langle X_i \rangle
\end{equation}
for some $\beta_i\in \bK$, $u_i,v_i\in \cM$ with $abc \succ u_i \fd (R_{l_i}) v_i$.
The corresponding composition in the presentation~\eqref{alg_Az} becomes
\begin{equation} \label{comp_Az}
	(R_j - \overline{R}_j) \circ_{b} (R_k - \overline{R}_k) =
	\big( \sum_{i=1}^{n}\beta_i  u_i (R_{l_i} - \overline{R}_{l_i}) v_i \big) + \Obs (R_j\circ_{b}R_k) \, , 
\end{equation}
where {\em an obstacle element} 
\begin{equation} \label{Obst}
	\Obs (R_j\circ_{b}R_k)   =
	\overline{R_k}a - \overline{R_j}c + \sum_{i=1}^{n}\beta_i\overline{R}_{l_i}u_iv_i \in \bZ\langle X_i \rangle
\end{equation}
appears. The obstacle $\Obs (R_j\circ_{b}R_k)$ is also a relation, a consequence of the relations in the presentation~\eqref{alg_Az}.
The notation $\Obs (R_j\circ_{b}R_k)$ is ambiguous because the obstacle depends on a particular choice of the right-hand side of the equation~\eqref{comp_A}.

The following statement is immediate because an obstacle is a relation
that can be rewritten to zero.
\begin{prop}\label{GSBasis} 
	Suppose that the presentation~\eqref{alg_A} is a Gr\"{o}bner-Shirshov basis.
	The presentation~\eqref{alg_Az} is a Gr\"{o}bner-Shirshov basis if and only if for each composition, one of its obstacles can be written as
	$$
	\Obs (R_j\circ_{b}R_k)   =
	\sum_{s=1}^{m}\gamma_s  u_s ({R_{l_s}} -\overline{R}_{l_s}) v_s \in \bZ \langle X_i \rangle
	$$
	for some $\gamma_s\in \bZ$, $u_s,v_s\in \cM$ with $abc \succ u_s \fd (R_{l_s}) v_s$.
\end{prop}

\section{Hopf Algebras}
\label{s2}

We now explore how the process of centrification interacts with Hopf-algebraic structures. Unless otherwise indicated, we incorporate by reference the Hopf-algebraic notation and definitions found in \cite{Mon}.

Throughout this section, we concern ourselves with the following setting. Let 
$$
\bH= \bK\langle X_i , i\in I | R_j=0, j\in \overline{J_0} \rangle
$$
be a Hopf algebra, and let 
$$
\bA= \bK\langle X_i , i\in I | R_j=0, j\in J \rangle
$$
be an algebra with, as before, $J=J_0\coprod\overline{J_0}$, such that the natural map $\bH\to\bA$ is an algebra homomorphism. 
Given an element $A\in \bK\langle X_i, i\in I\rangle$, consistent with Section~\ref{s1}, we denote
by $\overline{A}$ its image in $\bA^{c,J_0}$. 
Furthermore, we set $\widetilde{A}$ as its image in $\bH$ 
and $\widehat{A}$ as its image in $\bA$. We shall maintain this notation through the remainder of the paper.

\subsection{Quasicharacter Hopf Algebras}
Let us recall the notions of the group of group-like elements
$$
\cG = \cG(\bH) = \{ g\in\bH \,\mid\, \Delta (g)= g\otimes g \} 
$$
of a Hopf algebra $\bH$ and its set of skew-primitive elements
$$
\cP = \bigcup_{g,h\in\cG} \cP_{g,h},\,
\cP_{g,h}= \cP_{g,h}(\bH) = \{ A\in\bH \,\mid\, \Delta (A)= g\otimes A+ A \otimes h \}. 
$$
Let us call a skew-primitive element $A\in \cP_{g,h}$ {\em special}
if both $g$ and $h$ are central in $\bH$.
The following lemma is proved by a direct calculation: details are left to the reader.
\begin{lemma} \label{main_lemma}
	In a Hopf algebra $\bH$
	the following statements hold for elements $g,h,u,v\in \cG$, $X\in \cP_{g,h}$, $Y\in \cP_{u,v}$.
	\begin{enumerate}
		\item $uX \in \cP_{ug,uh}$.
		\item $Xu \in \cP_{gu,hu}$.
		\item If $gu=ug$ and $hu=uh$, then $[u,X] = uX-Xu \in \cP_{ug,uh}$.
		\item If $Y$ is special, then 
		$\Delta ([X,Y]) = gu\otimes [X,Y] + [X,Y] \otimes hv + [g,Y] \otimes Xv + Xu\otimes [h,Y]$. 
	\end{enumerate}
\end{lemma}

Let us call a Hopf algebra $\bH$ 
{\em a quasicharacter Hopf algebra} if
$\bH$ is generated as an algebra by its skew-primitive elements
$\cP$.
Observe that $g-h\in \cP_{g,h}$ so that $\cG$ is contained in the subalgebra generated by the skew-primitive elements in any Hopf algebra $\bH$. Our terminology is motivated by the notion of a character Hopf algebra \cite[1.5.2]{Kha},
where it is additionally required that $\cG$ is abelian and $\bH$ is generated by its $\cG$-semiinvariant skew-primitive elements. 

If the indexing set $J_0$ is well-ordered, we define $J_{0}^{<j} \coloneqq  \{s\in J_0 \,\mid\, s<j\} $ and 
$$
\bH_j = \bK\langle X_i , i\in I | R_k=0, k\in \overline{J_0}\cup J_{0}^{<j}\rangle.
$$
If each $R_j$, $j\in J_0$, gives a
special skew-primitive element of
the Hopf algebra $H_j$, it follows inductively via transfinite recursion that the kernels of the maps $\bH\rightarrow \bH_j$ are Hopf ideals so that all $\bH_j$ are Hopf algebras.

\begin{theorem} \label{Th1}
	Assume that $\bH$ is a quasicharacter Hopf algebra with the group $\cG = \cG (\bH)$, and that $\bA$ is also a Hopf algebra such that the map $\bH\to\bA$ is a homomorphism of Hopf algebras. Furthermore, assume that $J_0$ is well-ordered and each $R_j$, $j\in J_0$, gives a
	special skew-primitive element of
	the Hopf algebra $H_j$.
	Then the comultiplication of $\bH$ defines
	a Hopf algebra structure 
	$$
	\Delta (\overline{Y}) = \sum_{(Y)} \overline{Y}_{(1)}\otimes \overline{Y}_{(2)}
	$$
	on the centrification $\bA^{c,J_0}$
	and the natural map $\bA^{c,J_0}\rightarrow \bA$ is a homomorphism of Hopf algebras.
\end{theorem}

\begin{proof}
	Observe that $\bH$ is generated  by $\cG = \cG (\bH)$ and a collection of skew-primitive elements $\widetilde{Y}_k\in\cP_{g_k,h_k}$, $k\in K$. Thus, to move one centrification step up from $\bA^{c,J_{0}^{<j}}$ to $\bA^{c,J_{0}^{\leq j}}$, it suffices to request
	$\overline{R}_j$ to commute with all $g\in\cG$ and all $\overline{Y}_k$, $k\in K$.
	
	The Hopf algebra structure survives at each step because Lemma~\ref{main_lemma} ensures that all $g\overline{R}_j-\overline{R}_j g$ and $\overline{Y}_k\overline{R}_j-\overline{R}_j \overline{Y}_k$ generate an ideal of $\bA^{c,J_{0}^{<j}}$, which is a Hopf ideal.
	
	Now the standard transfinite recursion completes the proof.
	%
	%
\end{proof}

Let us call a primitive element $X\in \cP_{g,h}$ {\em right semispecial}
if $h$ is central in $\bH$. Similar to Theorem~\ref{Th1},
the reader can prove that the centrification of ``recursively'' right semispecial elements yields an $\bA$-comodule algebra. For convenience, we state a non-recursive version of this result without a proof.

\begin{prop} \label{Th1_4_Hopf_comod_alg}
	Assume that both algebras
	$$
	\bH= \bK\langle X_i , i\in I | R_j=0, j\in \overline{J_0} \rangle
	\longrightarrow 
	\bA= \bK\langle X_i , i\in I | R_j=0, j\in J \rangle
	$$
	are Hopf algebras, the map between them is a homomorphism of Hopf algebras 
	and $\bH$ is a quasicharacter Hopf algebra.
	Furthermore, we assume that $J_0$ is well-ordered and each $R_j$, $j\in J_0$ gives a
	right special skew-primitive element of $\bH$.
	Then the comultiplication of $\bH$ defines
	a right $\bA$-comodule algebra structure
	$$
	\rho (\overline{Y}) = \sum_{(Y)} \overline{Y}_{(1)}\otimes \widehat{Y}_{(2)}
	$$
	on the centrification $\bA^{c,J_0}$
	such that the natural map $\bA^{c,J_0}\rightarrow \bA$ is a homomorphism of $\bA$-comodule algebras, and $\bZ$ is a subalgebra
	generated by skew coinvariants $\overline{R}_j$, $j\in J_0$.
\end{prop}

Recall that a skew coinvariant is an element $Y\in\bA^{c,J_0}$ such that $\rho (Y) = Y \otimes A$ for some $A\in\bA$.
It follows that $A\in \cG(\bA)$.

\subsection{General Hopf Algebras}
Let us recall some generalities about extensions of Hopf algebras. Suppose that $\bH$ is a Hopf algebra, $(\bA,\rho)$ is a (right) $\bH$-comodule algebra, and $\bB$ is a subalgebra of $\bA$ with $\bA^{co \bH}=\bB$. We then call $\bB\subseteq \bA$ a (right) $\bH$-extension. We say that the extension $\bB\subseteq \bA$ is an $\bH$-Galois extension (or Hopf-Galois extension) if the natural map $$\bA\otimes_{\bB}\bA\to\bA\otimes_{\bK}\bH,\qquad A\otimes_{\bB} A'\mapsto (A\otimes 1)\rho(A')$$ is a linear isomorphism. The extension $\bB\subseteq\bA$ is called $\bH$-cleft if there exists an $\bH$-comodule map $\gamma:\bH\to\bA$ which is convolution invertible. All $\bH$-cleft extensions are $\bH$-Galois, but an $\bH$-Galois extension is cleft if and only if it satisfies the normal basis property - namely, that there exists an isomorphism of left $\bB$-modules and right $\bH$-comodules $\bA\cong\bB\otimes_{\bK}\bH$. See Theorem 8.2.4 in \cite{Mon} for details. Furthermore, an $\bH$-extension $\bB\subseteq\bA$ is $\bH$-cleft if and only if $\bA$ is a crossed product of $\bB$ with $\bH$ \cite[Ch. 7]{Mon}.


\begin{theorem}\label{Th2}
	Suppose that $\bH  = \bK\langle X_i , i\in I | R_j=0, j\in \overline{J_0} \rangle$
	is a Hopf algebra and that $$\Delta(R)\subseteq {{\pi}}^{-1}(Z(\bA^{c,J_0}))\otimes {{\pi}}^{-1}(Z(\bA^{c,J_0}))$$ for all $R\in\{\widetilde{R}_j\}_{j\in J_0}\subseteq\bH$, where ${{\pi}}$ is the natural map $\bH\to\bA^{c,J_0}$. Suppose also that $S(R)\in {{\pi}}^{-1}(Z(\bA^{c,J_0}))$ for all $R\in\{\widetilde{R}_j\}_{j\in J_0}\subseteq\bH$, where $S$ is the antipode of $\bH$. Then the centrification $\bA^{c,J_0}$ is a Hopf algebra over $\bK$.
\end{theorem}

\begin{proof}
	
	To show that $\bA^{c,J_0}$ is a Hopf algebra, it is enough to show that the kernel of the natural algebra homomorphism ${{\pi}}:\bH\to \bA^{c,J_0}$ is a Hopf ideal in $\bH$, i.e. that $\Delta(\ker({{\pi}}))\subseteq \ker({{\pi}})\otimes \bH + \bH\otimes \ker({{\pi}})$, $\epsilon(\ker({{\pi}}))=0$ and $S(\ker({{\pi}}))\subseteq \ker({{\pi}})$. Since $\ker({{\pi}})$ is an ideal in $\bH$, it is sufficient for the first condition to show that the images under $\Delta$ of some set of ideal-generators of $\ker({{\pi}})$ lie in $\ker({{\pi}})\otimes \bH + \bH\otimes \ker({{\pi}})$. It is straightforward to see that the ideal $\ker({{\pi}})$ is generated by the elements $[\widetilde{X}_i,\widetilde{R}_j]$ for $i\in I$ and $j\in J_0$.
	
	Now, note that if $R=\widetilde{R_j}$, for some $j\in J_0$, and $X\in\bH$ then 
	\begin{multline}\label{Eq1}\Delta([X,R])=\sum_{(X),(R)}X_{(1)}R_{(1)}\otimes X_{(2)}R_{(2)} - \sum_{(X),(R)}R_{(1)}X_{(1)}\otimes R_{(2)}X_{(2)} \\= \sum_{(X),(R)}[X_{(1)},R_{(1)}]\otimes X_{(2)}R_{(2)} + \sum_{(X),(R)}R_{(1)}X_{(1)}\otimes [X_{(2)},R_{(2)}]\in\bH\otimes\bH.\end{multline} By assumption, each $R_{(1)}$ and $R_{(2)}$ lie inside ${{\pi}}^{-1}(Z(\bA^{c,J_0}))$, and in particular the elements ${{\pi}}(R_{(1)})$ and ${{\pi}}(R_{(2)})$ are central in $\bA^{c,J_0}$. Thus, $${{\pi}}([X_{(1)},R_{(1)}])=[{{\pi}}(X_{(1)}),{{\pi}}(R_{(1)})]=0,$$ and similarly $${{\pi}}([X_{(2)},R_{(2)}])=[{{\pi}}(X_{(2)}),{{\pi}}(R_{(2)})]=0.$$ Therefore, $$\Delta([\widetilde{X}_i,\widetilde{R}_j])\in \ker({{\pi}})\otimes \bH + \bH\otimes \ker({{\pi}}),$$ and it follows that $\bA^{c,J_0}$ is a Hopf algebra, since it is easy to see that $\epsilon([X,R])=0$ and $S([X,R])\in \ker({{\pi}})$.
	
\end{proof}

\begin{theorem}\label{Th3}
	Suppose that $\bH  = \bK\langle X_i , i\in I | R_j=0, j\in \overline{J_0} \rangle$
	is a Hopf algebra. Let $\bS$ be the subalgebra generated by $\widetilde{R}_j\in\bH$, $j\in J_0$,
	$\bS^{+}\coloneqq \bS\cap\ker(\epsilon)$. Clearly $\pi(\bS)\subseteq Z(\bA^{c,J_0})$. Suppose that
	$$\Delta(R)\in \bS^{+}\otimes {{\pi}}^{-1}(Z(\bA^{c,J_0})) + {{\pi}}^{-1}(Z(\bA^{c,J_0}))\otimes \bS^{+}$$
	for all $R\in\{\widetilde{R}_j\}_{j\in J_0}\subseteq\bH$, where ${{\pi}}$ is the natural map $\bH\to\bA^{c,J_0}$. Suppose also that $S(R)\in \bS$ for all $R\in\{\widetilde{R}_j\}_{j\in J_0}\subseteq\bH$, where $S$ is the antipode of $\bH$, and that all the $\widetilde{R}_j$ lie in the kernel of the counit of $\bH$. Then $\bA$ is a Hopf algebra over $\bK$ and the natural map $\bA^{c,J_0}\rightarrow \bA$ is a homomorphism of Hopf algebras.
	
	If, in fact, $$\Delta(R)\in\bS^{+}\otimes \bS + {{\pi}}^{-1}(Z(\bA^{c,J_0}))\otimes \bS^{+}$$ for all $R\in\{\widetilde{R}_j\}_{j\in J_0}$, then $\bZ$ is a subalgebra of the coinvariants $(\bA^{c,J_0})^{co \bA}$.
\end{theorem}

\begin{proof}
	
	Denoting by $\widehat{\pi}$ the natural map $\bH\to\bA$, the assumptions guarantee that $\widehat{\pi}\otimes\widehat{\pi}(\Delta(R))=0$ for all $R\in\{\widetilde{R}_j\}_{j\in J_0}$. Since the set $\{\widetilde{R}_j\}_{j\in J_0}$ generates the ideal $\ker(\widehat{\pi})$, and the assumptions give that $\epsilon(\widetilde{R}_j)=0$ and $S(\widetilde{R}_j)\in\ker(\widehat{\pi})$ for each $j\in J_0$, it follows that $\bA$ is a Hopf algebra over $\bK$. 
	
	It is easy to see the assumptions guarantee that $\bA^{c,J_0}$ is a Hopf algebra, using Theorem~\ref{Th2}, and similarly straightforward to see that the natural map $\bA^{c,J_0}\rightarrow \bA$ is a homomorphism of Hopf algebras.
	
	Finally, suppose $$\Delta(R)\in \bS^{+}\otimes \bS + {{\pi}}^{-1}(Z(\bA^{c,J_0}))\otimes \bS^{+}$$ for all $R\in\{\widetilde{R}_j\}_{j\in J_0}$. We observe that, as $\bK$-vector spaces,
	$$
	{{\pi}}^{-1}(Z(\bA^{c,J_0}))=\bK\oplus {{\pi}}^{-1}(Z(\bA^{c,J_0}))^{+} \, ,
	$$
	where ${{\pi}}^{-1}(Z(\bA^{c,J_0}))^{+}\coloneqq {{\pi}}^{-1}(Z(\bA^{c,J_0}))\cap\ker(\epsilon)$, and similarly $\bS=\bK\oplus\bS^{+}$. Then, for $R=\widetilde{R}_{j}$, $j\in J_0$, we obtain that 
	$$\Delta(R)\in 1\otimes\bS^{+} + \bS^{+}\otimes 1 + {{\pi}}^{-1}(Z(\bA^{c,J_0}))\otimes \bS^{+}.$$ Using the fact that in a coalgebra we have the identity $H=\sum_{(H)}H_{(1)}\epsilon(H_{(2)})=\sum_{(H)}\epsilon(H_{(1)})H_{(2)}$, we can in fact conclude that $$\Delta(R)\in R\otimes 1 + 1\otimes R + {{\pi}}^{-1}(Z(\bA^{c,J_0}))\otimes \bS^{+}.$$
	This expression also shows that, viewing $\bA^{c,J_0}$ as an $\bA$-comodule algebra via $$\bA^{c,J_0}\xrightarrow{\Delta}\bA^{c,J_0}\otimes \bA^{c,J_0}\twoheadrightarrow \bA^{c,J_0}\otimes \bA$$ we get
	$$\rho(R)= R \otimes 1\in \bA^{c,J_0}\otimes\bA$$
	for all $R\in\widetilde{\bS}$. 
	Hence, we have that $\bZ\subseteq (\bA^{c,J_0})^{co \bA}$.
\end{proof}

%

Using Theorems~\ref{Th2} and \ref{Th3}, we obtain the following corollaries, of which the first is the most powerful.
It yields a Hopf-Galois extension with central invariants that were previously studied \cite{Rum}.

\begin{cor}\label{sHpf}
	Suppose that $\bH  = \bK\langle X_i , i\in I | R_j=0, j\in \overline{J_0} \rangle$
	is a Hopf algebra and the subalgebra $\bS$ generated by $\widetilde{R}_j\in\bH$, $j\in J_0$, is a Hopf subalgebra. Then the following results hold:
	\begin{enumerate}
		\item The centrification $\bA^{c,J_0}$ is a Hopf algebra over $\bK$. 
	\end{enumerate}
	From now on, assume that the $\widetilde{R}_j$ lie in the kernel of the counit of $\bH$.
	\begin{enumerate}[resume]
		\item The algebra $\bA$ is a Hopf algebra over $\bK$. 
		\item The natural map $\bA^{c,J_0}\rightarrow \bA$ is a homomorphism of Hopf algebras. 
		\item The subalgebra $\bZ$ is a subalgebra of $(\bA^{c,J_0})^{co \bA}$. 
	\end{enumerate}
	For the final two results, assume that $\bA^{c,J_0}$ has cocommutative coradical.
	\begin{enumerate}[resume]
		\item The inclusion $\bZ\subseteq(\bA^{c,J_0})^{co \bA}$ is, in fact, an equality. 
		\item The inclusion $\bZ\subseteq\bA^{c,J_0}$ is an $\bA$-Galois extension.
	\end{enumerate}	
\end{cor}

\begin{proof}
	All that remains for (1), (2), (3) and (4) is to check that $$\Delta(R)\in \bS^{+}\otimes {{\pi}}^{-1}(Z(\bA^{c,J_0})) + {{\pi}}^{-1}(Z(\bA^{c,J_0}))\otimes \bS^{+}$$ for $R\in\{\widetilde{R}_j\}_{j\in J_0}\subseteq\bH$, in the case when the $\widetilde{R}_j$ lie in the kernel of the counit of $\bH$. Since $\bS$ is a Hopf subalgebra of $\bH$, we see immediately that $\Delta(R)\in\bS\otimes\bS$.
	
	Since $\bS=\bK\oplus\bS^{+}$, we conclude that, for $R\in\{\widetilde{R}_j\}_{j\in J_0}$, we have $$\Delta(R)\in\lambda 1\otimes 1 + \bK\otimes \bS^{+} + \bS^{+}\otimes\bK + \bS^{+}\otimes\bS^{+}$$ for some $\lambda\in\bK$. Then $$0=\Delta(\epsilon(R))=\epsilon\otimes\epsilon(\Delta(R))= \lambda\epsilon(1)^2=\lambda$$ and we conclude that $\lambda=0$. As $\bS^{+}$ clearly lies inside ${{\pi}}^{-1}(Z(\bA^{c,J_0}))$, the result follows.
	
	For (5) and (6), note that the assumption that $\bS$ is a Hopf subalgebra of $\bH$ guarantees that $\bZ$ is Hopf subalgebra of $\bA^{c,J_0}$. In particular, as it is central, it is a normal Hopf subalgebra of $\bA^{c,J_0}$. By Remark 1.1(4) in \cite{Schn} (cf. \cite[Th. 3.1]{Tak}), this implies that $\bA^{c,J_0}$ is faithfully flat as a $\bZ$-module and the inclusion $\bZ\subseteq\bA^{c,J_0}$ is an $\bA^{c,J_0}/\bZ^{+}\bA^{c,J_0}$-Galois extension (recalling here that $\bZ^{+}$ is the kernel of the counit of $\bZ$).
	Since the relations $\widetilde{R}_j$ for $j\in J_0$ all lie in the kernel of the counit in $\bH$, it follows that $\bZ^{+}\bA^{c,J_0}$ is the ideal of $\bA^{c,J_0}$ generated by $\widetilde{R}_j$, $j\in J_0$. In other words, the projection $\bA^{c,J_0}\to \bA^{c,J_0}/\bZ^{+}\bA^{c,J_0}$ is precisely the Hopf algebra homomorphism $\bA^{c,J_0}\to\bA$, and the result follows.
	
\end{proof}
%
%

\begin{cor}
	Suppose that $\bH  = \bK\langle X_i , i\in I | R_j=0, j\in \overline{J_0} \rangle$
	is a Hopf algebra. Let $\bS$ be the subalgebra of $\bH$ generated by the set $\{\widetilde{R}_j\}_{j\in J_0}\subseteq \bH$, and let
	$\bT$ 
	be the subalgebra of $\bH$ generated by $\bS$ and $Z(\bH)$. Suppose that
	$$\Delta(R)\in \bT\otimes \bT$$
	for all $R\in\{\widetilde{R}_j\}_{j\in J_0}\subseteq\bH$. Suppose also that $S(R)\in \bT$ for all $R\in\{\widetilde{R}_j\}_{j\in J_0}\subseteq\bH$, where $S$ is the antipode of $\bH$. Then the centrification $\bA^{c,J_0}$ is a Hopf algebra over $\bK$.
\end{cor}

The following proposition shows how our assumptions can change if we only want $\bA^{c,J_0}$ to be an $\bA$-comodule algebra.

\begin{prop}
	Suppose that $\bH  = \bK\langle X_i , i\in I | R_j=0, j\in \overline{J_0} \rangle$
	is a Hopf algebra. Let $\bI_{J_0}$ be the ideal of $\bH$ generated by the set $\{\widetilde{R}_j\}_{j\in J_0}\subseteq \bH$ and $\bS$ be the subalgebra generated by these same elements. Suppose further that $$\Delta(R)\in \bH\otimes \bI_{J_0} + \bS^{+}\otimes {{\pi}}^{-1}(Z(\bA^{c,J_0}))$$ for all $R\in\{\widetilde{R}_j\}_{j\in J_0}\subseteq\bH$, the $\widetilde{R}_j$ lie in the kernel of the counit of $\bH$, and $S(R)\in\bS$ for all $R\in\{\widetilde{R}_j\}_{j\in J_0}$, where $S$ is the antipode of $\bH$. Then $\bA$ is a Hopf algebra over $\bK$, $\bA^{c,J_0}$ is an $\bA$-comodule algebra, and the natural map $\bA^{c,J_0}\rightarrow \bA$ is a homomorphism of $\bA$-comodule algebras.
	
	Finally, if in fact $$\Delta(R)\in \bH\otimes \bI_{J_0} + \bS^{+}\otimes \bS$$ for all $R\in\{\widetilde{R}_j\}_{j\in J_0}\subseteq\bH$, then $\bZ$ is a subalgebra of the coinvariants $(\bA^{c,J_0})^{co \bA}$.
\end{prop}

\begin{proof}
	
	The majority of the proof works in much the same way as in the proof of Theorem~\ref{Th3}. All that is different is showing that $\bA^{c,J_0}$ is an $\bA$-comodule algebra.
	
	To show this, it is enough to show that the map $\bH\xrightarrow{\Delta}\bH\otimes\bH \twoheadrightarrow\bA^{c,J_0}\otimes\bA$ sends elements of the form $[X,R]$ to zero, where the $X$ is in some generating set of $\bH$ and $R\in\{\widetilde{R}_j\}_{j\in J_0}\subseteq \bH$. Under our assumptions, this is clear from Equation~\ref{Eq1}.
	
\end{proof}


We already saw in Corollary~\ref{sHpf} one condition for the extension $\bA^{c,J_0}\to\bA$ to be Hopf-Galois. Let us now discuss some further properties of the extensions $\bA^{c,J_0}\to\bA$.

\begin{prop}\label{GalCond}
	Suppose that both of the algebras
	$$
	\bH= \bK\langle X_i , i\in I | R_j=0, j\in \overline{J_0} \rangle
	\longrightarrow 
	\bA= \bK\langle X_i , i\in I | R_j=0, j\in J \rangle
	$$
	are Hopf algebras and the map between them is a homomorphism of Hopf algebras.
	Suppose further that $\bA^{c,J_0}$ is an $\bA$-comodule algebra such that the natural map $\bA^{c,J_0}\rightarrow \bA$ is a homomorphism of Hopf-comodule algebras, and $\bZ$ is a subalgebra of the coinvariants $(\bA^{c,J_0})^{co \bA}$. Then the extension $\bA^{c,J_0}\to\bA$ satisfies the following properties.
	\begin{enumerate}
		\item If the extension $\bH\to\bA$ is $\bA$-cleft, then $\bA^{c,J_0}\to\bA$ is $\bA$-cleft.
		\item If the antipode $S$ satisfies $S^2=id$ and there exists a $\bK$-subspace $\bV\subseteq\bH$ such that the restriction $\bV\to\bA$ is a linear isomorphism and such that $\bV$ is a subcoalgebra of $\bH$, then $\bA^{c,J_0}\to\bA$ is $\bA$-cleft.
		\item If $\bA$ is finite-dimensional, then the extension $\bA^{c,J_0}\to\bA$ is $\bA$-Galois.
	\end{enumerate}
\end{prop}

\begin{proof}
	(1) Suppose that the extension $\bH\to\bA$ is $\bA$-cleft. Then, by definition, there exists convolution-invertible $\bA$-comodule map $\gamma:\bA\to\bH$.
	Recall that ${{\pi}} :  \bH\to\bA^{c,J_0}$ is the natural map.
	It is clear that ${{\pi}}\circ\gamma:\bA\to\bA^{c,J_0}$ is an $\bA$-comodule map. 
	All that remains is to check that it is convolution invertible. This follows from the convolution calculation: 
	\begin{multline*}({{\pi}}\circ\gamma)\ast({{\pi}}\circ\gamma^{-1})(A)=\sum_{(A)}{{\pi}}(\gamma(A_{(1)})){{\pi}}(\gamma^{-1}(A_{(2)}))
		\\={{\pi}}(\gamma\ast\gamma^{-1}(A))={{\pi}}(\epsilon(A))=\epsilon(A).
	\end{multline*}
	(2) Let $V_m$, $m\in M$, be a $\bK$-basis of the vector space $\bV$. Letting $\widehat{\pi}$ be the natural map $\bH\to\bA$, we shall then denote $\widehat{V_m}\coloneqq\widehat{\pi}(V_m)$, so that $\widehat{V_m}$, $m\in M$, is a $\bK$-basis of $\bA$. Define $\gamma:\bA\to\bA^{c,J_0}$ to be the linear map defined on the given basis as $$\gamma(\widehat{V_m})={{\pi}}(V_m)$$ and $\widehat{\gamma}:\bA\to\bH$ to be the linear map defined by $$\widehat{\gamma}(\widehat{V_m})=V_m.$$  We would like to know if $\gamma$ is an $\bA$-comodule map and if it is convolution-invertible. 
	
	Let us first show that $\gamma$ is convolution-invertible with convolution inverse $\pi S\widehat{\gamma}$. It is enough to show this on our given basis of $\bA$, as follows:
	\begin{align*}
		\gamma\ast\pi S\widehat{\gamma}(\widehat{V_m}) = &\sum_{(\widehat{V_m})}\gamma(\widehat{V_{m,(1)}})\pi S\widehat{\gamma}(\widehat{V_{m,(2)}}) \stackrel{\mbox{\tiny 2}}{=}
		\\ \sum_{(V_m)}\gamma\widehat{\pi}(V_{m,(1)})\pi S\widehat{\gamma}\widehat{\pi}(V_{m,(2)}) \stackrel{\mbox{\tiny 3}}{=} & \; \pi\left(\sum_{(V_m)}\widehat{\gamma}\widehat{\pi}(V_{m,(1)}) S\widehat{\gamma}\widehat{\pi}(V_{m,(2)})\right)
		\\ \stackrel{\mbox{\tiny 4}}{=} \pi\left(\sum_{(V_m)}V_{m,(1)} S(V_{m,(2)})\right) = & \; \pi\varepsilon_{\bH}(V_m) \stackrel{\mbox{\tiny 6}}{=} \varepsilon_{\bA}(\widehat{V_m})1_{\bA^{c,J_0}}.
	\end{align*}	
	Here, the second and sixth equality follow from $\widehat{\pi}$ being a homomorphism of Hopf algebras, the third from the fact that $\gamma=\pi\widehat{\gamma}$, and the fourth from the fact that $\bV$ is a subcoalgebra in $\bH$ (so each $V_{m,(1)}$ and $V_{m,(2)}$ are a $\bK$-linear combination of some $V_m$, $m\in M$) and that $\widehat{\gamma}\widehat{\pi}(V_m)=V_m$ for all $m\in M$.  A similar argument holds for $\pi S\widehat{\gamma}\ast\gamma$.
	
	Now, we just need to ask whether $$\rho\circ\gamma(\widehat{V_m})=(\gamma\otimes id)\circ\Delta_{\bA}(\widehat{V_m})$$ for all $m\in M$. This follows from a similar calculation as the one used for the convolution invertibility. 
	
	
	(3) If $\bA$ is finite-dimensional, then we only need to show that the map $$\alpha:\bA^{c,J_0}\otimes_{\bB}\bA^{c,J_0}\to\bA^{c,J_0}\otimes_{\bK}\bA,\quad A\otimes_{\bB}A'\mapsto(A\otimes 1)\rho(A')$$ is surjective \cite[8.3.1]{Mon}, where $\bB\coloneqq(\bA^{c,J_0})^{co \bA}$. This follows by considering the commutative diagram 
	
	$$
	\xymatrix{
		\bH\otimes_{\bK} \bH \ar@{->>}[rr]^{\beta} \ar@{->}[d]^{{{\pi}}\otimes_{\bB}{{\pi}}} & & \bH\otimes_{\bK} \bH \ar@{->}[d]^{{{\pi}}\otimes\widehat{\pi}} & & \\
		\bA^{c,J_0}\otimes_{\bB}\bA^{c,J_0} \ar@{->}[rr]^{\alpha} & & \bA^{c,J_0}\otimes_{\bK}\bA & &\\
	}
	$$
	
	where the top row is the map $\beta:H\otimes H'\mapsto (H\otimes 1)\Delta(H')$. In particular, the vertical maps are surjective since they are obtained from projections, and the top row is surjective since the identity $\bH\to\bH$ is an $\bH$-Galois extension (it is clearly $\bH$-cleft) with $\bH^{co \bH}=\bK$. Hence, $\bA^{c,J_0}\otimes_{\bB}\bA^{c,J_0}\to\bA^{c,J_0}\otimes_{\bK}\bA$ is surjective.

\end{proof}



\section{Free Associative Hopf Algebra}
\label{s3}
In this section we consider the free associative algebra\footnote{We henceforth omit the indexing set when no confusion is likely to arise, in order to avoid cluttering up equations.} $\bK \langle X_i \rangle$ with the standard Hopf algebra structure
$$
\Delta (X_i)= 1\otimes X_i+ X_i \otimes 1 \, .
$$
We further assume that all $R_j$ are primitive.
It follows that all $[R_i,R_j]$ are also primitive (Lemma~\ref{main_lemma})
and $\bH$, $\bA$ and $\bA^{c,J_0}$ are all Hopf algebras (Theorem~\ref{Th1}). We may also conclude that $\bZ=\bA^{c,J_0}$ and that $\bA^{c,J_0}\to\bA$ is an $\bA$-Galois extension (Corollary~\ref{sHpf}).

\subsection{Lie algebras}
\label{lie_alg}
If $\bK$ is a field of characteristic zero, the primitive elements of $\bK \langle X_i \rangle$
are precisely elements of the free Lie algebra $\Lie \langle X_i \rangle$ over $\bK$. Thus, our relations
define three Lie algebras and Lie algebra homomorphisms
$$
\Lie \langle X_i \rangle
\rightarrow
\fh = \Lie \langle X_i  | R_j  \rangle
\rightarrow
\fa^{c,J_0} = \Lie \langle X_i | [X_i,R_k], R_j \rangle 
\rightarrow
\fa=\Lie \langle X_i | R_l \rangle
$$
(where $j\in \overline{J_0}$,  $k\in {J_0}$,  $l\in {J}$)
so that $\fa^{c,J_0}$ is a central extension of $\fa$ and 
our associative algebras are their universal enveloping algebras: 
$$
\bK \langle X_i \rangle = U (\Lie \langle X_i \rangle)
\rightarrow
\bH = U(\fh)
\rightarrow
\bA^{c,J_0} = U(\fa^{c,J_0})
\rightarrow
\bA = U(\fa) \, .
$$
Let $\bK$ be a field of positive characteristic $p$.
These considerations remain valid as soon as all the relations $R_i$ 
are elements of the free Lie algebra $\Lie \langle X_i \rangle$

Let us contemplate arbitrary primitive relations in positive characteristic.
The primitive elements of $\bK \langle X_i \rangle$
are precisely elements of the free restricted Lie algebra $\Lip \langle X_i \rangle$ over $\bK$. Thus, our relations
define three restricted Lie algebras and their homomorphisms
\begin{multline*}
\Lip \langle X_i \rangle
\rightarrow
\fh = \Lip \langle X_i  | R_j  \rangle
\\ \rightarrow
\fa^{c,J_0} = \Lip \langle X_i | [X_i,R_k], R_j \rangle 
\rightarrow
\fa=\Lip \langle X_i | R_l \rangle
\end{multline*}
so that $\fa^{c,J_0}$ is a central extension of $\fa$ and 
our associative algebras are their restricted enveloping algebras: 
$$
\bK \langle X_i \rangle = U_0 (\Lip \langle X_i \rangle)
\rightarrow
\bH = U_0(\fh)
\rightarrow
\bA^{c,J_0} = U_0(\fa^{c,J_0})
\rightarrow
\bA = U_0(\fa) \, .
$$

\subsection{Versal central extension}
Let $X_i$ be a basis of a Lie algebra $\fg$ over $\bK$.
In our previous notation, consider
$$
\bH = \bK\langle X_i\rangle \, , \ 
\bA = U(\fg) = \bK\langle X_i | R_{i,j} \rangle
$$
where ${R_{i,j}}\coloneqq X_iX_j-X_jX_i-[X_i,X_j]$ for all $i<j$. It is straightforward to see that the conditions of Corollary~\ref{sHpf} and Proposition~\ref{GalCond}(2) are satisfied, so we immediately obtain (writing $\bA^z$ for $\bA^{z,J}$, where $J$ is the full indexing set of relations) that $\bA^z\to \bA$ is $\bA$-cleft. In particular, we get that $\bA^z$ is a crossed product of $\bZ$ with $U(\fg)$.
In fact, the discussion in Section \ref{lie_alg} reveals that its full centrification $\bA^z$ will be (in any characteristic) a universal enveloping algebra $U(\widehat{\fg})$
of some central extension $\varpi:\widehat{\fg}\rightarrow \fg$. As above, $\widehat{\fg}$ has a Lie algebra presentation 
$\Lie \langle X_i | [X_i,R_{j,k}] \rangle$,
so there is a canonical splitting $\fg\to\widehat{\fg}$ induced by sending $X_i\mapsto X_i$. This leads 
to the natural bilinear form (cf. \cite[Ex. 7.7.5]{W}),
$$ \lambda: \fg \times \fg \rightarrow \ker (\varpi),
\ (X_i, X_j)\mapsto R_{i,j}.$$
Since $\varpi:\widehat{\fg}\rightarrow \fg$ is a central extension of Lie algebras,  
the form $\lambda$ is a 2-cocycle with coefficients in the trivial $\fg$-module. This means that 
\begin{equation} \label{cocycle_cond}
  R_{i,j} + R_{j,i} = 0\ \mbox{ and } \
  \lambda ([X_i,X_j] ,  X_k)  + \lambda ([X_j,X_k], X_i)  +  \lambda ([X_k,X_i] ,  X_j) = 0
\end{equation}  
for all $i$, $j$ and $k$. 
It follows that
\begin{equation} \label{cocycle_cond_2}
\bZ \cong \bK [R_{i,j}] / (\mbox{equations~\eqref{cocycle_cond}}) \cong \Sym (\wedge^2 \fg) / (\mbox{cocycle conditions})
\end{equation}
so that $\bZ$ can be identified with the polynomial functions on $Z^2 (\fg,\bK)$ -- for example,
$R_{i,j}$ is the function sending a cocycle $\mu\in Z^2 (\fg,\bK)$ to $\mu(X_i,X_j)$.
Notice the ``cocycle conditions'' in \eqref{cocycle_cond_2} are the second equations in \eqref{cocycle_cond} for all $i,j,k$:
\begin{equation*}
[X_i,X_j] \wedge  X_k  + [X_j,X_k] \wedge X_i  +  [X_k,X_i] \wedge X_j = 0
\end{equation*} 
Thus, $\widehat{\fg}$ is a ``versal'' central extension of $\fg$, whose kernel is $Z^2 (\fg,\bK)^\ast$. Indeed (cf. \cite[Th. 7.9.2]{W}), if $\fg$ is perfect, then the universal central extension of $\fg$ is $[\widehat{\fg},\widehat{\fg}]\to\fg$, and for any other central extension $\fe\to\fg$ the unique map $[\widehat{\fg},\widehat{\fg}]\to \fe$ coming from the universal property is simply a restriction of a (non-necessarily-unique) map $\widehat{\fg}\to\fe$.

\subsection{Universal enveloping algebra from reduced enveloping algebra}
\label{ug_from_ug}
Suppose that $(\fg,\,^{[p]})$ is a restricted Lie algebra over a field $\bK$ of positive characteristic $p$.
Given a linear form $\chi\in\fg^{*}$ we can give a presentation of the restricted enveloping algebra $U_\chi(\fg)$ as follows. If $I$ is an indexing set of a basis $\{X_i\}_{i\in I}$ of $\fg$, then within the free associative Hopf algebra $\bK\langle X_i,\vert\,i\in I \rangle$ we define the elements
$$
{R_{i,k}} =  X_iX_k-X_kX_i-[X_i,X_k], \
{R_i}\coloneqq X_i^p-X_i^{[p]}-\chi(X_i)^p
$$
for $i,k\in I$. 
Now, set $J_0=\{(i,k)\,\vert\, i<k\}\subset I\times I$, $J_1=I$ and $J=J_0\coprod J_1$. Then 
$$U_\chi(\fg)=\langle X_i, i\in I\,\vert\, {R_j},j\in J\rangle.$$

Using centrality of the elements $X_k^p-X_k^{[p]}$ in $U(\fg)$, it is straightforward to conclude that
$$
U_\chi(\fg)^{c,J_1}=U(\fg)
\mbox{ and }
\bZ=\bK[X_k^p-X_k^{[p]}]_{k\in I},
$$
the latter known as the $p$-centre of $U(\fg)$. We can further observe that $U_\chi(\fg)^{c,J_0}$ is a free $\bZ$-module, and if $\vert I\vert=n<\infty$, it has finite rank $p^{n}$. 

Since the elements ${R_{i,k}}$ and ${R_k}$ for $i,k\in I$ are primitive in $\bK\langle X_i,\vert\,i\in I \rangle$, Theorem \ref{Th2} leads to the unsurprising result that $U(\fg)$ is a Hopf algebra. If $\chi=0$, then all the conditions of Corollary~\ref{sHpf} are satisfied, and we obtain the similarly expected result that $U_0(\fg)$ is a Hopf algebra, the natural map $U(\fg)\to U_0(\fg)$ is a homomorphism of Hopf algebras, $\bZ$ is equal to the subalgebra of coinvariants $U(\fg)^{co U_0(\fg)}$, and the extension $U(\fg)^{co U_0(\fg)}\to U(\fg)$ is $U(\fg)$-cleft. Furthermore, it is straightforward to check that the conditions for Proposition~\ref{GalCond}(2) hold, and so the extension $U(\fg)^{co U_0(\fg)}\to U(\fg)$ is in fact $U(\fg)$-cleft.


\section{Askey-Wilson algebra} 
\label{s4}

One existing appearance of centrifications in the literature comes from the Askey-Wilson polynomials. The Askey-Wilson algebra $AW$ was first introduced by Zhedanov in \cite{Z} in order to facilitate a study of these polynomials. Fixing $q\in \bK$ with $q^4\neq 1$, and structure constants $\underline{s}=(b,c_0,c_1,d_0,d_1)\in\bK^5$,
we may define the following three relations in $\bK\langle X,Y,Z\rangle$:
$${R}_X\coloneqq qYZ-q^{-1}ZY - bY - c_0X - d_0,$$
$${R}_Y\coloneqq qZX-q^{-1}XZ - bX - c_1Y - d_1,$$
$${R}_Z\coloneqq qXY-q^{-1}YX -Z.$$
The Askey-Wilson algebra $AW_q(\underline{s})$ is defined as
\begin{equation} \label{AW1}
	AW_q(\underline{s}) \coloneqq
	\bK\langle X,Y,Z\,\vert\, {R}_X, {R}_Y, {R}_Z\rangle \, .
\end{equation}
There is a different version 
of the Askey-Wilson algebra \cite{KZ}.
Specifically, given an element $0\neq q\in\bK$ such that $q^4\neq 1$ and three elements $a,b,c\in\bK$, we define
another version of the Askey-Wilson algebra: 
\begin{equation} \label{AW2}
	AW_q(a,b,c) \coloneqq \bK\langle A,B,C\,\vert {R}_A,{R}_B,{R}_C\rangle
	\ \mbox{ where }
\end{equation}   
$${R}_{A}\coloneqq A + \dfrac{qBC-q^{-1}CB}{q^2-q^{-2}} - \dfrac{a}{q+q^{-1}},$$
$${R}_{B}\coloneqq B + \dfrac{qCA-q^{-1}AC}{q^2-q^{-2}} - \dfrac{b}{q+q^{-1}},$$
$${R}_{C}\coloneqq C + \dfrac{qAB-q^{-1}BA}{q^2-q^{-2}} - \dfrac{c}{q+q^{-1}}.$$

Later on Terwilliger introduced the \emph{universal Askey-Wilson algebra} $\Delta$ \cite{T}.
This is the full centrification of the presentation~\eqref{AW2} of the Askey-Wilson algebra.
It is clear that this algebra does not depend on the choice of $a,b,c$, although it does still depend on $q$. Equipping the set $\{A,B,C\}$ with the ordering $A\succ B\succ C$, it is straightforward to check that
$\bK\langle A,B,C\,\vert {R}_A,{R}_B,{R}_C\rangle$ is a Gr\"{o}bner-Shirshov basis for $AW$, and one can further calculate
that the only possible obstacle vanishes:
$$\Obs({R}_C\circ_{B}{R}_A)=0.$$
By Proposition~\ref{GSBasis}, the following is a Gr\"{o}bner-Shirshov basis over $\bZ$:  
$$
\Delta = AW_q(a,b,c)^z=\bZ\langle A,B,C\,\vert\, {R}_A-\overline{R_A}, {R}_B-\overline{R_B}, {R}_C-\overline{R_C}\rangle \, .
$$

\section{Bannai-Ito algebra and anticommutator spin algebra}
\label{s5}

The \emph{anticommutator spin algebra}
is the following algebra \cite{AK,GP}:
\begin{equation} \label{AS}
	AS \coloneqq \bK\langle X,Y,Z\,\vert {R}_X,{R}_Y,{R}_Z\rangle
	\ \mbox{ where }
\end{equation}   
$$
{R}_X\coloneqq YZ+ZY-X, \ 
{R}_Y\coloneqq XZ+ZX - Y, \ 
{R}_Z\coloneqq XY+YX-Z.$$


This is a special case of a more general construction: the \emph{Bannai-Ito algebra}. These were introduced in \cite{TVZ} as a tool for understanding the Bannai-Ito polynomials -- a type of polynomial which can be interpreted as a $q\to -1$ limit of Askey-Wilson polynomials. The Bannai-Ito algebra depends on a triple $\omega=(\omega_1,\omega_2,\omega_3)\in\bK^3$:
$$
BI^\omega\coloneqq \bK\langle X,Y,Z\,\vert {R}_X-\omega_1,{R}_Y-\omega_2,{R}_Z-\omega_3\rangle \, .
$$

From our point of view, the full centrification of the presentation~\eqref{AS} is {\em the universal Bannai-Ito algebra}.
Using the ordering $X\succ Y\succ Z$, it is straightforward to check that
$\bK\langle X,Y,Z\,\vert {R}_X,{R}_Y,{R}_Z\rangle$ is a Gr\"{o}bner-Shirshov basis for $AS$, and one can further calculate
that the only possible obstacle vanishes:
$$\Obs({R}_Z\circ_{Y}{R}_X)=0.$$
Thus,
$$AS^z=\bZ\langle X,Y,Z\,\vert\, {R}_X-\overline{R_X}, {R}_Y-\overline{R_Y}, {R}_Z-\overline{R_Z}\rangle$$ is a Gr\"{o}bner-Shirshov basis over $\bZ$, using Proposition~\ref{GSBasis}. One can further observe that $\bZ$ is the polynomial algebra
$\bK [\overline{R_X}, \overline{R_Y}, \overline{R_Z} ]$ and that the Bannai-Ito algebra $BI^\omega$ is isomorphic to $AS^z \otimes_{\bZ}\bK$ where the homomorphism $\bZ\rightarrow \bK$ is given by $\omega$:
$$
\overline{R_X}\mapsto \omega_1, \
\overline{R_Y}\mapsto \omega_2, \
\overline{R_Z}\mapsto \omega_3 \, .
$$

\end{document}